\documentclass[a4paper,10pt]{article}
\usepackage{graphicx} 
\usepackage[utf8]{inputenc}   
\usepackage[T1]{fontenc}
\usepackage{lmodern}
\usepackage{comment}

\usepackage{subfigure}
\usepackage{float}
\usepackage{natbib}

\usepackage{a4wide}
\usepackage{indentfirst}
\usepackage{authblk}

\usepackage{amsmath,amssymb}
\usepackage{amsthm}
\usepackage{hyperref}

\usepackage[margin=1.2in]{geometry}
\usepackage{graphicx}
\usepackage{listings}
\usepackage[utf8]{inputenc}
\usepackage{appendix}
\numberwithin{equation}{section}
\newtheorem{teo}{Theorem}[section]
\newtheorem*{teo*}{Theorem}
\newtheorem*{prop*}{Proposition}
\newtheorem*{corol*}{Corollary}
\newtheorem{prop}[teo]{Proposition}

\newtheorem{corol}[teo]{Corollary}
\newtheorem{lema}[teo]{Lemma}
\newtheorem{defi}[teo]{Definition}
\newtheorem{hyp}{Hypothesis}

\theoremstyle{definition}

\newtheorem{ex}[teo]{Example}

\newcommand{\Dom}{\operatorname{Dom}}

\newcommand{\D}{\mathbb{D}}
\newcommand{\E}{\mathbb{E}}
\newcommand{\F}{\mathcal{F}}
\renewcommand{\H}{\mathcal{H}}

\newcommand{\N}{\mathbb{N}}
\renewcommand{\P}{\mathbb{P}}

\newcommand{\R}{\mathbb{R}}

\newcommand{\1}{\mathbf{1}}
\newcommand{\phy}{\varphi}
\newcommand{\eps}{\varepsilon}

\title{ ON THE SHORT-TIME BEHAVIOUR OF UP-AND-IN BARRIER OPTIONS USING MALLIAVIN CALCULUS}

\author[1]{Òscar Burés}

\affil[1]{ Universitat de Barcelona, Departament de Matemàtica Econòmica, Financera i Actuarial. \authorcr Diagonal 690--696, 08034 Barcelona, Spain.}
\date{\today}

\begin{document}

\maketitle
\begin{abstract}
In this paper we study the short-maturity asymptotics of up-and-in barrier options under a broad class of stochastic volatility models. Our approach uses Malliavin calculus techniques, typically used for linear stochastic partial differential equations, to analyse the law of the supremum of the log-price process. We derive a concentration inequality and explicit bounds on the density of the supremum in terms of the time to maturity. These results yield an upper bound on the asymptotic decay rate of up-and-in barrier option prices as maturity vanishes. We further demonstrate the applicability of our framework to the rough Bergomi model and validate the theoretical results with numerical experiments.
\end{abstract}



\textbf{Word count:} 5462.

\section{Introduction}

A key problem in Quantitative Finance is the pricing of financial derivatives—contracts between a buyer and a seller whose value depends on an underlying asset. Among all types of derivatives, particular attention from both academics and practitioners has been drawn to options. These contracts give the holder the right, but not the obligation, to exercise the contract if market conditions at maturity are favourable.

There are several types of options, with European call and put options being among the most prominent. For these instruments, the price is typically computed as the expected value, under a risk-neutral probability measure, of the positive difference between the asset’s value at maturity and the strike price. Since these options depend only on the terminal value of the underlying asset, they allow for analytical pricing under simple modeling assumptions, such as the Black-Scholes framework.

However, real-world financial markets offer a much broader variety of options that are more complex to analyze. Among these are path-dependent options, whose payoff depends on the entire trajectory of the underlying asset up to maturity. Notable examples include Asian options, lookback options, Bermuda options, and barrier options—the latter being the focus of this article.

Barrier options are similar to European options in that their payoff structure is identical, but with the crucial distinction that the payoff is contingent upon the underlying asset reaching a predefined level, known as the barrier. Among barrier options, we distinguish between "In" options, which are activated only if the asset price hits the barrier during the contract's life, and "Out" options, which are deactivated if the barrier is breached. In this article, we focus specifically on the asymptotic behaviour of up-and-in barrier call options, which are call options that can only be exercised if the asset price reaches a barrier level that lies above its initial value. We focus on up-and-in barrier options for two main reasons. The first one, since Malliavin calculus allows us to deal with the supremum of a stochastic process, we can apply those techniques to "up" barrier options (which depend on the maximum value of the stock price on a time window). Then, we choose to focus on up-and-in barrier options because, as we will show in the paper, the premium of these options tend to zero as the maturity attains small values regardless of the strike value. The asymptotic behaviour for up-and-out barrier options can be deduced combining the main result of this paper, Theorem \ref{th: main theorem} and the main result from \cite{jafari2025option}.

Closed-form pricing formulas for up-and-in barrier call options exist under the Black-Scholes model. However, empirical evidence shows that the assumption of constant volatility is often insufficient to capture certain behaviors of the underlying asset, such as volatility clustering, leverage effects, or smile/skew phenomena as it is shown in e.g. \cite{hullwhite}, \cite{wiggins87}, \cite{steinstein}, and \cite{heston93}. Moreover, the work of authors in \cite{CR}, \cite{ALV}, and \cite{F},  \cite{gatheralJR18} and \cite{bayerfrizgatheral16} show that the asset price dynamics is more compatible with rough volatility models. When stochastic volatility is introduced, the analytical treatment of barrier options becomes considerably more challenging. Nevertheless, it is still possible to study the properties of such options using tools such as stochastic calculus, Malliavin calculus, and computational methods like Monte Carlo simulations. The research on barrier option focuses on their pricing and hedging (see, for instance \cite{zvan2000pde}, \cite{brown2001robust} or \cite{kou2003pricing}). Previous work on the asymptotic behavior of barrier options can be found in the literature. In \cite{CarradaHerrera2013} the authors develop an expansion for double barrier options with constant volatility and discontinuities coming from a compound Poisson process. Regarding the analysis of the asymptotic behaviour under stochastic volatility, the work of the authors in \cite{Hu16062010} and \cite{kato2013asymptotic} provides asymptotic expansions of barrier options under specific stochastic volatility models such as the CEV or SABR volatility models.

The difficulty of analysing the supremum of a stochastic process that is not Gaussian makes it difficult to extend the papers previously cited to a general stochastic volatility framework. Using Malliavin calculus we are able to show that the probability of attaining the Barrier (and therefore, the price of an up-and-in Barrier call option) decays faster than any polynomial of the maturity.

In \cite{jafari2025option}, it is shown that for European options in the out-of-the-money (OTM) and at-the-money (ATM) cases, the option price tends to zero as maturity $T \to 0$, whereas in the in-the-money (ITM) case, the price tends to the difference between the initial asset price and the strike. Moreover, it is proven that this convergence rate is, at most, polynomial in $T$. Since the exercise condition of an up-and-in barrier call option is more restrictive than that of a standard European call option, it follows that its price is always lower. Consequently, in the OTM and ATM cases, the price of an up-and-in barrier call option also vanishes as 
$T \to 0$. In this paper, we address the problem of determining the limit value of an ITM up-and-in barrier call option as maturity vanishes, and we further show that the rate at which this limit is reached is faster than any polynomial in $T$, thus reaching its limit value in a faster rate as the observed for European options.

To study the asymptotic behaviour of up-and-in barrier call options, we employ Malliavin calculus to analyse the probability that the underlying asset hits the barrier. Specifically, we derive two types of bounds for this probability. The first is a direct bound on the probability that the supremum of the underlying process exceeds the barrier, using a concentration inequality. The second involves estimating the CDF of the supremum by first obtaining estimates of its density function. In both cases, the approach is far from straightforward. Indeed, the supremum of the underlying process does not satisfy the regularity conditions typically required to apply Malliavin-based concentration inequalities or density criteria.

To overcome this difficulty, we rely on the work of \cite{nualart1988continuite}, which provides a framework for computing the Malliavin derivative of the supremum. This allows us to apply a suitable concentration inequality to our setting. Furthermore, we draw upon techniques similar to those in \cite{florit1995local} to prove the existence of a density for the supremum and derive estimates for it. The problem of analysing the density of the supremum of a process has been previously studied in the context of solutions to stochastic differential equations—for instance, in \cite{pu2018stochastic} or \cite{dalang2020density}, where the authors derive density estimates for the supremum of the solution to the stochastic heat equation with additive noise. In our case, the noise is not additive, as we assume the presence of stochastic volatility. However, we manage to treat the problem in a way that is effectively equivalent to having additive noise, provided that the volatility process is not completely correlated with the driving noise of the stock price.

This paper is organized as follows. Section \ref{sec 2} introduces the model assumed for the underlying asset. Section \ref{sec 3} presents the Malliavin calculus tools used to study the density of the asset price supremum over $[0,T]$. In Section \ref{sec 4}, we state the main result: up-and-in barrier call option prices vanish as $T \to 0$, with a convergence rate faster than any polynomial. Section \ref{sec 5} proves this limit by introducing an auxiliary random variable controlling the supremum, inspired by techniques from \cite{florit1995local}. Section \ref{sec 6} provides a first estimate for the barrier-hitting probability using a concentration inequality, for which we verify the required Malliavin differentiability conditions as in \cite{nourdin2009density} using techniques similar to the ones in \cite{nualart1988continuite}. In Section \ref{sec 7}, we prove the existence of a density for the supremum using a local criterion, derive estimates for it, and obtain a bound for its CDF—confirming the faster-than-polynomial decay. Finally, Section \ref{sec 8} applies our results to the Rough Bergomi model, illustrating the decay rate both theoretically and numerically.

\section{The Model} \label{sec 2}
 We assume the following model for the log-price process, under a risk-neutral measure given by the market and under the assumption of null interest rate $r = 0$:  
\begin{eqnarray}\label{model}
X_t=x-\frac{1}{2}\int^t_0\sigma^2_s ds+\int^t_0\sigma_s(\rho dW_s+\sqrt{1-\rho^2}dB_s),\,\, t\in [0,T],
\end{eqnarray}
where $X_0$ is the current log price, $W$ and $B$ are independent standard Brownian motions and $\rho \in (-1,1)$. The volatility process $\sigma$ is assumed to be a square-integrable process, adapted to the filtration generated by $W$. Throughout the paper we will use the notation $Z := \rho W + \sqrt{1-\rho^2}B$, so Model \eqref{model} can also be written as
\[
X_t=x-\frac{1}{2}\int^t_0\sigma^2_s ds+\int^t_0\sigma_sdZ_s,\,\, t\in [0,T].
\]
We denote by ${\mathbb{F}}^W = \{\F_t^W; t \in [0,T]\}$ and ${\mathbb{F}}^B = \{\F^B_t; t \in [0,T]\}$ the natural filtrations generated by the independent processes $W$, $B$. Moreover, we denote by $\mathbb{F} = \{\F_t; t \in [0,T]\}$ the filtration generated by both $W$ and $B$, that is, $\F_t = \F^W_t \vee \F^B_t$ for all $t \in [0,T]$. It is well known that the pricing formula for a European call option with a strike price $K$ under Model \eqref{model} is given by  
\begin{equation*}
C_{t}=E\left[(e^{X_{T}}-K)_{+}| {\cal F}_{t}\right].
\end{equation*} 
Given a real number $B > S_0$, an up-and-in barrier call option with underlying asset $S$, strike $K$, maturity $T$ and barrier $B$ is a financial derivative with payoff $(S_T - K)_+ \1_{\{\sup_{t \in [0,T]} S_t \geq B\}}$. In other words, the price of an up-and-in barrier call option under this configuration at time $t =0$ is given by
\[
C_0^{b} = \E\left[(S_T - K)_+ \1_{\{\sup_{t \in [0,T]} S_t \geq B\}} \right].
\]
We are interested in the short-time behaviour of $C_0^b$, so we can assume without loss of generality that $T \leq 1$. Since $S_t = \exp(X_t)$ and the exponential function is increasing, we can rewrite the premium of the barrier option in terms of the log-price. I.e. the price of such barrier option at time $t = 0$ can be written as
\[
C_0^{b} = \E\left[(e^{X_T} - K)_+ \1_{\{M_T\geq b\}} \right], \quad M_T = \sup_{t \in [0,T]} X_t, \quad b = \log B.
\]
From the fact that $ 0 \leq \1_{\{M_T \geq b\}} \leq 1$ it is clear that $C^{b}_0 \leq C_0$. This also can be interpreted in a financial sense: since the conditions for an up-and-in barrier call option to be exercised are more restrictive than the conditions for an European call option, one potentially gets more benefit from the European call option than from the barrier option, so the European call option must be more expensive. Notice that a simple application of Hölder's inequality shows that
\[
C_0^b \leq ||(S_T -K)_+||_{L^p(\Omega)} \P\left(M_T \geq b \right)^{1/q}, \quad \frac{1}{p} + \frac{1}{q} = 1.
\]
So, if we manage to prove that $\lim_{T \to 0} C_0^b = 0$, then, an analysis of the speed of convergence of the term $\P\left(M_T \geq b \right)$ to zero will give us an upper bound for speed of convergence of $C_0^b$ to zero. Hence, the objectives of this paper are the following:
\begin{itemize}
    \item Show that $\lim_{T \to 0} C_0^b = 0$.
    \item Using a concentration inequality we will show an a-priori upper bound for $\P(M_T \geq b)$.
    \item Using more involved Malliavin calculus techniques, we will give an upper bound for the density function of $M_T$ in order to estimate the cummulative distribution function of $M_T$ and derive and alternative upper bound for $\P(M_T \geq b)$.
\end{itemize}
\section{Malliavin Calculus Tools} \label{sec 3}
In this section we will introduce all the notions from Malliavin Calculus needed in order to cover the problem of estimating the law of the supremum of a stochastic process of the form \eqref{model}. Let $\H = L^2([0,T])$ and let $B := \{B(h); h \in \H\}$ the associated isonormal Gaussian process. We denote by $\mathcal{S}$ the class of smooth random variables of the form
\[
F = f(B(h_1), \dots, B(h_n))
\]
where $h_i \in \H$ for all $i \in \{1, \dots, n\}$ and $f \in \mathcal{C}^{\infty}_b(\R^n)$. Given a random variable $F \in \mathcal{S}$, we define the Malliavin derivative of $F$ as the $\H$-valued stochastic process $DF = \{D_t F; t \in [0,T]\}$ where
\[
D_tF = \sum_{i=1}^n \partial_i f(B(h_1), \dots, B(h_n)) h_i(t).
\]
We can also define higher order derivatives in a similar manner. Indeed, for any natural number $k \geq 1$ we define the $k$-th order Malliavin derivative of $F$ as the $\H^{\otimes k}$-valued stochastic process $D^k F = \{ D_z F; z \in [0,T]^k\}$ where
\[
D^k_z F = \sum_{i=1}^n \frac{\partial}{\partial z_1} \cdots \frac{\partial}{\partial z_k} f(B(h_1), \dots, B(h_n)) (h_1 \otimes \cdots \otimes h_k)(z), \quad z = (z_1, \dots, z_k).
\]
It is well known that the operators $D^k$ are closable from $L^p(\Omega)$ to $L^p(\Omega; \H^{\otimes k})$ for all $p \geq 1$, $k \geq 1$. This allows to define the spaces $\D^{k,p}$ as the closure of $\mathcal{S}$ with respect to the semi-norm $||\cdot||_{k,p}$ defined by
\[
||F||_{k,p} = \left \{ \E[|F|^p] + \sum_{j=1}^k \E[ ||DF||_{\H^{\otimes j}}^p] \right \}^{1/p}.
\]
We define also $\D^{\infty} := \bigcap_{k \geq 1} \bigcap_{p \geq 1} \D^{k,p}$.

The Malliavin derivative $D$ on the space $L^2(\Omega)$ has an adjoint operator called the divergence operator or Skorohod integral and it is represented by $\delta$. For any process $u \in L^2(\Omega; \H)$, the element $\delta(u)$ is uniquely determined by the duality relationship
\[
\E[ F\delta(u)] = \E\left[\int_0^T D_t F \cdot u_t  dt \right], \quad \text{ for all } F \in \D^{1,2}.
\]
Among all the properties of the Skorohod integral, we will be using the following relationship in Section \ref{sec 7}.
\begin{prop} \label{prop: integral of product}
    Let $F \in \D^{1,2}$ and let $u \in \Dom(\delta)$. Then,
    \[
    \delta(Fu) = F\delta(u) - \langle DF, u\rangle,
    \]
    where $\langle \cdot, \cdot \rangle$ denotes the usual scalar product in $L^2([0,T])$.
\end{prop}
One of the main applications of the Malliavin calculus to probability theory is the existence of criteria for checking the absolute continuity of the law of random variables and deciding whether the density function of $F$ is smooth. The most classical criterion in order to prove that the density of a random vector is smooth is the following.
\begin{teo}
    Let $F = (F_1, \dots, F_d)$ be a random vector whose components are in $\D^{\infty}$. Assume that the Malliavin matrix $\gamma_{i,j} := \langle DF^i, DF^j\rangle$ satisfies $|\det \gamma|^{-1} \in L^p(\Omega)$ for all $p  \geq 1$. Then, the random vector $F$ is absolutely continuous with respect to the Lebesgue measure and $F$ possesses an infinitely differentiable density.
\end{teo}
\begin{proof}
    See \cite{nualart2006malliavin}.
\end{proof}
This result is widely applied in the study of the density of solutions to SDEs and SPDEs since, under the assumption that the coefficients of the differential equations are infinitely many times differentiable, one can prove in most cases that the solution belongs to $\D^{\infty}$ and under ellipticity conditions on the diffusion coefficient one can check that the determinant of the inverse of the Malliavin matrix has moments of all orders, concluding that the density of the solution is smooth.

In this paper, we study the supremum of the stock price process $S$ over the time interval $[0, T]$. It is well known that the functional
$$
M_T := \sup_{t \in [0,T]} X_t
$$
belongs to $\mathbb{D}^{1,2}$, but does not exhibit higher regularity. Although $M_T$ lacks the smoothness required to directly apply the previously stated criterion, a localized version of the result is available and well-suited for this setting.
\begin{teo}
    Let $F = (F^1, \dots, F^d)$ be a random vector whose components are in $\D^{1,2}$. Let $A$ be an open set of $\R^d$. Assume that there exist $\H$-valued random variables $u_A$ and a $d \times d$ random matrix $\gamma_A = (\gamma_{i,j})_{1\leq i,j \leq d}$ such that
    \begin{itemize}
        \item [1.] $u^j_A \in \D^{\infty}(\H)$ for all $j \in \{1, \dots, d\}$.
        \item [2.] $\gamma_{i,j} \in \D^{\infty}$ for all $(i,j) \in \{1, \dots, d\}^2$ and $|\det \gamma_A|^{-1} \in L^p(\Omega)$ for all $p \geq 1$.
        \item [3.] $\langle DF^i, u_A^j\rangle  = \gamma_A^{i,j}$ on the set $\{F \in A\}$ for all $(i,j) \in \{1, \dots, d\}^2$.
    \end{itemize}
    Then, the random vector $F$ possesses an infinitely differentiable density on the open set $A$.
\end{teo}
\begin{proof}
    See \cite{florit1995local}.
\end{proof}
\section{Main Result} \label{sec 4}
In this section we will state the main result of this paper, which describes the short-time behaviour of the price of an up-and-in barrier call option. The following assumptions for the volatility process $\sigma$ are needed.
\begin{hyp} \label{hyp 1}
    There exist $0 < \alpha < \beta$ such that
    \[
    \alpha \leq \sigma_t \leq \beta
    \]
    for almost every $t \in [0,T]$.
\end{hyp}
\begin{hyp} \label{hyp 2}
    For all $p \geq 2$, $\sigma \in \mathbb{L}^{1,p}_W$.
\end{hyp}
From now on, we will assume that the volatility process always satisfies Hypotheses \eqref{hyp 1} and \eqref{hyp 2}. Hence, when we assume that $X$ follows Model \eqref{model} we are also implying the previous hypotheses on the volatility process. 

In the present work, we apply Malliavin calculus techniques in order to understand the asymptotic behaviour of the probability that the barrier is achieved in order to give an upper bound for the asymptotic behaviour of $C_0^b$. The main result is stated as follows.
\begin{teo} \label{th: main theorem}
    Let $X$ be the log-price of an asset $S$ with dynamics given by \eqref{model}. Let $B > S_0$ and let $K$ be the strike price. Then, an up-and-in barrier call option of the form $$C_0^b = \E\left[ (S_T-K)_+ \1_{\{\sup_{t \in [0,T]} S_t \geq B\}}\right]$$ exhibits the following short-time behaviour:
    \begin{itemize}
        \item[(i)] $\lim_{T \to 0} C_0^b = 0$.
        \item[(ii)] There exists $T_0 > 0$ and constants $c_1, c_2, c_3 > 0$ such that the probability of reaching the barrier satisfies
        \[
        \P \left( \sup_{t \in [0,T]} S_t \geq B \right) \leq \min \left\{\exp\left( -\frac{(b-x)^2}{c_1 T} \right),  \int_{b}^{\infty}\frac{c_2}{\sqrt{T}} \exp\left( -\frac{(z-x)^2}{2c_3 T}\right) dz\right \}
        \]
        for $T \leq T_0$.
    \end{itemize}
    In particular, $C_0^b = o(T^{\alpha})$ for every $\alpha> 0$.
\end{teo}
Observe that Theorem \ref{th: main theorem} states that the price of an up-and-in barrier option decays faster than any polynomial as $T\to 0$. In \cite{jafari2025option}, one can deduce that the price of a European call option tends to zero at least at a polynomial speed depending on $H$. In this paper we explicitly prove that up-and-in barrier options tend to zero faster than any polynomial.

From now on, we will use $C$ to denote a positive real constant not depending on $T$ that may differ from line to line but its specific value is not important for the main conclusion of this article.

\section{Short-Time Behaviour Of An Up-And-In Barrier Call Option} \label{sec 5}
In this section we aim to cover the first of the objectives regarding the study of the asymptotic behaviour of an up-and-in barrier call option. We will follow closely the ideas in \cite{dalang2020density} and we adapt them for a process of the form \eqref{model}. We follow the arguments from \cite{nualart1988continuite}, \cite{florit1995local} and \cite{nualart2006malliavin}. The strategy is based on the fact that even it is not true that $\sup_{t \in [0,T]} X_t \in \D^{\infty}$, there exist a random variable defined via the Hölder norm of $X$ that belongs to $\D^{\infty}$ and controls $\sup_{t \in [0,T]} X_t$. We will rely on this random variable to show that $\lim_{T \to 0} \P(M_T \geq b) = 0$ and conclude from there that $\lim_{T \to 0} C_0^b= 0$. Then, once we know that every up-and-in barrier call option tends to zero, we can focus on the speed of convergence. 

In order to prove the first objective, that is, proving that $\lim_{T \to 0} \P(M_T \geq b) = 0$ we need to know the Hölder regularity of the sample paths of $X_t$ in order to define the auxiliary random variables controlling the supremum of $X$. The regularity of the paths of $X$ can be easily derived applying Kolmogorov's continuity criterion, as it is shown in the following result.
\begin{lema}
    Let $X$ be defined as in \eqref{model}. The sample paths of $X$ are $\gamma$-Hölder continuous for $\gamma \in \left(0, \frac{1}{2} \right)$. In particular, for every $p \geq 1$ there exists a constant $C_p > 0$ not depending on $T$ such that
    \[
    \E[|X_t - X_s|^p] \leq C_p|t-s|^{p/2}
    \]
\end{lema}
\begin{proof}
    Notice that, since $\sigma$ is a process with continuous sample paths, the map
    \[
    t \mapsto -\frac{1}{2}\int_0^t \sigma_s^2 ds
    \]
    is indeed $\mathcal{C}^1([0,T])$ almost surely. Moreover,
    \[
    \left|\frac{1}{2}\int_s^t \sigma_u^2 du \right|^{p} \leq C|t-s|^{p} \leq C|t-s|^{p/2}
    \]
    because $T \leq 1$. Hence, the Hölder continuity of the paths of $X$ is inherited from the Hölder continuity of the paths of 
    \[
    \int_0^t \sigma_s dZ_s.
    \]
    Now, Burkholder's inequality shows that
    \[
    \E\left[ \left|\int_s^t \sigma_u dZ_u \right|^p\right] \leq c_p \E\left[ \left( \int_s^t \sigma_u^2 du\right)^{p/2}\right] \leq c|t-s|^{p/2}.
    \]
    Kolmogorov's continuity criterion shows that the sample paths of the stochastic integral of $\sigma$ with respect to $Z$ are of $\gamma$-Hölder continuous with $\gamma \in \left(0, \frac{1}{2}-\frac{1}{p}\right)$ for all $p \geq 1$. Letting $p \to \infty$ we conclude the result.
\end{proof}
Notice that up to this point, one can show that $\lim_{T \to 0} \P(M_T \geq b) = 0$ using a continuity argument. Nevertheless, since the definition of the auxiliary random variables controlling $M_T$ is not only useful to prove that $\lim_{T \to 0} \P(M_T \geq b) = 0$ but it is also key in the estimation of the density of $M_T$ we will prove that $\lim_{T \to 0} \P(M_T \geq b) = 0$ in a more involved way that uses the Hölder norm of $X$. The definition of the random variable that controls $M_T$ is the following.
\begin{defi}
    Consider  $p_0 \in \N$ and let $\gamma_0 \in \R$ such that $p_0- 2 > \gamma_0 > 4$. We define $Y = \{Y_r; r \in [0,T]\}$ as
\[
Y_r := \int_{[0,r]^2} \frac{(X_t - X_s)^{2p_0}}{|t-s|^{\gamma_0}} dt ds.
\]
\end{defi}
Before discussing the relation between $Y$ and $M_T$ we shall check that indeed the process $Y$ is well defined.
\begin{lema}
 For any $r \in [0,T]$, $Y_r \in \bigcap_{p \geq 1} L^p(\Omega)$. In particular, for every $p \geq 1$ there exists a constant $C > 0$ not depending on $T$ such that
    \[
    \E[|Y_r|^p] \leq C r^{2p} T^{(p_0 - \gamma_0)p}.
    \]
\end{lema}
\begin{proof}
    Hölder's inequality applied to $p \geq 1$ shows that
    \begin{align*}
    \E[|Y_r|^p] \leq &r^{2(p-1)} \int_{[0,r]^2} \frac{\E[|(X_t - X_s)^{2p_0 p}]}{|t-s|^{\gamma_0 p }} dt ds \\
    \leq & C r^{2(p-1)}\int_{[0,r]^2} \frac{|t-s|^{p_0p}}{|t-s|^{\gamma_0 p}} dt ds \\
    \leq &C r^{2p} T^{(p_0 - \gamma_0)p},
    \end{align*}
    as desired.
\end{proof}
The relation between the process $Y$ and the supremum of $X$ relies on the fact that the cumulative distribution function of $Y_T$ is bounded by the cumulative distribution function of $M_T$. Moreover, since $Y_T$ is more regular than $M_T$, it is easier to deduce properties on $Y_T$ and transfer them to $M_T$. The relation between the cumulative distribution function of both random variables is a consequence of the Garsia-Rodemich-Rumsey lemma.
\begin{prop}
    For every $\beta > 0$ there exists a constant $C > 0$ such that if $$Y_r \leq C \beta^{2p_0} T^{\frac{4-\gamma_0}{2}},$$ then 
    $$\sup_{t \in [0,r]} |X_t - x|\leq \beta.$$
\end{prop}
\begin{proof}
     We apply the Garsia-Rodemich-Rumsey lemma as in \cite[Lemma A.3.1]{nualart2006malliavin} or \cite[Proposition A.1]{dalang2007hitting} with the following configuration
    \begin{align*}
        &S = [0,r], \quad \rho(t,s) = |t-s|^{1/2}, \quad \mu(dt) = dt \\
        & \Psi(x) = x^{2p_0}, \quad p(x) = x^{\gamma_0 / 2p_0}, \quad f = X
    \end{align*}
    in order to obtain that
    \begin{align*}
    |X_t - X_s| \leq &8 \sup_{x \in [0,r]} \int_0^{2 \rho(t,s)} \Psi^{-1} \left( \frac{Y_r}{\mu(B_{\rho}(x, u/2))^{2}}\right) p(du) \\ 
    \leq & 8 \int_0^{2 \rho(t,s)} \frac{Y^{1/2p_0}}{\mu(B_{\rho}(s, u/2))^{2/p_0}} u^{\frac{\gamma_0}{2p_0}-1} du \\ 
    \leq & C Y_r^{\frac{1}{2p_0}} \int_0^{2\rho(t,s)} u^{\frac{-2}{p_0}}u^{\frac{\gamma_0}{2p_0} -1} du \\
    \leq & CY_r^{\frac{1}{2p_0}} |t-s|^{\frac{\gamma_0 - 4}{4p_0}}.
    \end{align*}
    Therefore, if we let $Y_r \leq C^{-2p_0} {\beta}^{2p_0} T^{\frac{4-\gamma_0}{2}}$ then
    \[
    |X_t - X_s| \leq C \cdot C^{-1} \beta T^{\frac{4-\gamma_0}{4p_0}} |t-s|^{\frac{\gamma_0 - 4}{4p_0}} \leq \beta.
    \]
    The result now follows from taking $s = 0$ and supremums on both sides of the previous inequality.
\end{proof}
We can also deduce the following immediate consequence.
\begin{corol} \label{dominator}
    For every $\beta > x$ there exists a constant $C > 0$ such that if
    \[
    Y_r \leq R_T(\beta) := C (\beta-x)^{2p_0} T^{\frac{4-\gamma_0}{2}} 
    \]
    then
    \[
    \sup_{t \in [0,r]} X_t\leq \beta.
    \]

\end{corol}
\begin{proof}
    Let $C > 0$ the constant such that if $ Y_r \leq C (\beta-x)^{2p_0} T^{\frac{4-\gamma_0}{2}}$ then $\sup_{t \in [0,r]}|X_t-x|\leq \beta-x$. Its existence is ensured by the previous proposition. Now, if $Y_r \leq  C (\beta-x)^{2p_0} T^{\frac{4-\gamma_0}{2}} $ then $\sup_{t \in [0,r]} X_t \leq \beta$ as claimed.
\end{proof}
The dynamics of both $X$ and $M_T$ are hard to study, however, the sample paths of $Y$ are monotonically increasing. This, together with Corollary \ref{dominator} allows us to find the limit as the maturity tends to zero of an up-and-in barrier call option.
\begin{prop}
    Let $C_0^b$ a an up-and-in barrier call option with underlying $S$, strike $K$, maturity $T$ and barrier $B > S_0$. Then,
    \[
    \lim_{T \to 0} C_0^b = 0.
    \]
\end{prop}
\begin{proof}
    On the one hand, if the associated European call option is Out-The-Money (OTM) then, its premium $C_0$ tends to zero as $T$ tends to zero and therefore the result is deduced from the fact that $C_0^b \leq C_0$. For the ATM and ITM cases we will study $\P(M_T\geq b)$ where $b = \log B$ denotes the log-barrier. On the one hand, Corollary \ref{dominator} shows that
    \[
    \P(M_T \geq b) \leq \P(Y_T \geq R_T) = \P(Y_T - R_T \geq 0).
    \]
    On the one hand, the sample paths of $Y$ are continuous and monotonically decreasing with $Y_0 = 0$. On the other hand, since $R_T(b) = C (b-x)^{2p_0} T^{\frac{4-\gamma_0}{2}}$ then it holds that $R_T(b) \to \infty$ as $T \to 0$. Hence, $\lim_{T \to 0} Y_T - R_T(b) = -\infty$ almost surely (and therefore, the limit also holds in probability) concluding that
    \[
    \lim_{T \to 0} \P(Y_T \geq R_T) = 0.
    \]
    Consider now the random variable $\1_{\{M_T \geq b\}}$. Since $\P(M_T \geq b) \to 0$ then $\1_{\{M_T \geq b\}} \to 0$ in $L^1(\Omega)$. Thus, $\1_{\{M_T \geq b\}} \to 0$ in probability. Moreover, since $M_T$ is increasing in $T$, one has that for $T_1 \leq T_2$
    \[
    \1_{\{M_{T_1} \geq b\}} \leq \1_{\{M_{T_2} \geq b\}}
    \]
    so the sequence of random variables $\1_{\{M_T \geq b\}}$ is decreasing as $T \to 0$. Thus, the convergence $\1_{\{M_T \geq b\}} \to 0$ also holds almost surely. Since $(S_T - K)_+ \in L^1(\Omega)$ and $(S_T-K)_+\1_{\{M_T \geq b\}} \to 0$ almost surely we can conclude by the dominated convergence theorem that
    \[
    \lim_{T \to 0} \E[(S_T-K)_+\1_{\{M_T \geq b\}}] = \lim_{T \to 0} C_0^b = 0
    \]
    as we wanted to show.
\end{proof}
\section{A First Asymptotic Result} \label{sec 6}
Now that we know that $\lim_{T \to 0} C_0^b = 0$ we want to analyse how fast does $\P(M_T \geq b)$ tend to zero as $T \to 0$. A first and direct way to estimate the speed of convergence of this term is via a concentration inequality. In other words, given $z \in (x, \infty)$ we want to get an estimate of the tail probability 
\[
\P(M_T \geq z)
\]
and apply it to $z = b$. The literature on concentration inequalities is wide so we can easily find results in the literature that adapt to our situation. The result we will rely on in order to estimate the tail probability of $M_T$ is the following.
\begin{lema} \label{tail}
    Let $F \in \D^{1,2}$ such that $||DF||_{\H} \leq C$ for some constant $C > 0$ almost surely. Then,  for every $\theta \geq 0$,
    \[
    \E\left[ e^{\theta F}\right] \leq e^{\E[F]\theta + (1/2)C^2 \theta ^2}.
    \]
    In particular, for every $\lambda \geq 0$,
    \[
    \P(F - \E[F] \geq \lambda) \leq \exp\left( - \frac{\lambda^2}{2C^2} \right).
    \]
\end{lema}
\begin{proof}
    See for instance \cite{nourdin2009density} or  \cite{ustunel2010analysis} (Theorem 9.1.1).
\end{proof}
In order to apply this generic concentration inequality to $M_T$ we have to show that $M_T \in \D^{1,2}$ and $||DM_T||_{\H} \leq C$ for some constant $C > 0$ almost surely. The following technical results deal with the Malliavin regularity of $M_T$ and the computation of its Malliavin derivative. From now on, we will make use of the fact that $\rho \in (-1,1)$ and we will perform all the Malliavin calculus techniques with respect to the Brownian motion $B$ that, recall, is independent from the Brownian motion $W$ that drives the volatility, and therefore $B$ is independent of the volatility process $\sigma$. Hence, without loss of generality we will denote by $D$ the Malliavin derivative with respect to $B$, $\delta$ the Skorohod integral with respect to $B$ and $\D^{k,p}$ the Malliavin spaces under the operator $D$.

The following lemma will help us to prove the regularity of $M_T$ and find an almost explicit expression of its Malliavin derivative.
\begin{lema}
    The process $X$ attains its maximum in the interval $[0,T]$ at a unique random point $\tau \in [0,T]$.
\end{lema}
\begin{proof}
    The argument uses Lemma 2.6 in \cite{kim1990cube}, which states that if $\{Z_t; t \in [0,T]\}$ is a Gaussian process with $Var(Z_t - Z_s) \neq 0$ for $s \neq t$ then the supremum of $Z$ is attained at a unique point almost surely. In particular, it is shown that for two pair of distinct points $t_1, t_2 \in [0,T]$ then
    \[
    \P\left(\sup_{t \in N_1} Z_t  = \sup_{t \in N_2} Z_t \right) = 0
    \]
    for every two neighbourhoods $N_1, N_2$ of $t_1$ and $t_2$ respectively. Now, observe that the process $X$ conditioned to $W$ is a Gaussian process, so
    \[
    \P\left(\sup_{t \in N_1} X_t  = \sup_{t \in N_2} X_t  \bigg| \mathcal{F}^W_T \right) = 0.
    \]
    which implies that
    \[
    \P\left(\sup_{t \in N_1} X_t  = \sup_{t \in N_2} X_t \right) = 0
    \]
    for all possible neighbourhoods $N_1$ and $N_2$ of $t_1$, $t_2$ and every couple of points $t_1 \neq t_2$, concluding this way that the supremum is attained with probability 1 at a unique point.
\end{proof}
The fact that the supremum is attained at a unique point allows us to compute the Malliavin derivative of $M_T$ as a function of the almost sure supremum.
\begin{lema}

The random variable $M_T$ belongs to $\D^{1,2}$ and
    \[
    D_{\cdot}M_T = \sigma_{\cdot} \sqrt{1-\rho^2} \1_{[0,\tau]}(\cdot).
    \]
\end{lema}
\begin{proof}
    Consider $\{t_n; n \geq 1\}$ a dense subset of $[0,T]$. Define
    \[
    M_T^n = \max \{X_{t_1}, \dots, X_{t_n} \}.
    \]
    Then, $\lim_{n \to \infty} M_T^n = M_T$ and $\lim_{n \to \infty} \tau_n = \tau$ almost surely. Using that the operator $D$ is local, as it is shown in \cite{nualart2006malliavin}, (Section 1.3.5), we have that
    \[
    DM_T^n = \sigma_{\cdot} \sqrt{1-\rho^2} \1_{[0,\tau_n]}(\cdot)
    \]
    where $\tau_n$ is such that $M_T^n = X_{\tau_n}$. On the one hand, since $|M_T^n - M_T|^2 \to 0$ almost surely and the sequence $|M_T - M_T^n|^2$ is monotonically decreasing we find, by the monotone convergence theorem that $M_T^n \to M_T$ in $L^2(\Omega)$. Moreover,
    \[
    \E\left[ || DM_T^n ||^2_{\H}\right] = \E\left[\int_0^{\tau_n}  \sigma^2_t (1-\rho^2) dt\right] 
    \]
    Hence, since $\tau_n \leq T$ we have that
    \[
    \sup_{n \geq 1} \E\left[ || DM_T^n ||^2_{\H}\right] < \infty.
    \]
    Finally, let $\mathcal{M}_{\cdot} = \sigma_{\cdot} \sqrt{1-\rho^2} \1_{[0,\tau]}(\cdot)$. for every $u \in L^2(\Omega; \H)$ we find using the dominated convergence theorem that
    \[
    \lim_{n \to \infty} \E\left[ \int_0^{\tau_n} \sqrt{1-\rho^2}  \sigma_t u_t dt\ \right] = \E\left[ \int_0^{\tau} \sqrt{1-\rho^2}  \sigma_t u_t dt\ \right]
    \]
    so $D M^n_T$ converges to $\mathcal{M}$ in the weak topology of $L^2(\Omega; \H)$. This allows us to conclude, thanks to \cite{nualart2006malliavin} (Lemma 1.2.3) that $M_T \in \D^{1,2}_B$ and $D_{\cdot}M_T = \sigma_{\cdot} \sqrt{1-\rho^2} \1_{[0,\tau]}(\cdot).$ as desired.
\end{proof}
From this lemmas we are able to construct a tail estimate of the probability of reaching the barrier.
\begin{teo}[Tail estimate] Let $M_T = \sup_{t \in [0,T]} X_t$. Then there exists $T_0$ such that for $T \leq T_0$ the following tail estimate holds:
    \[
    \P\left(M_T \geq b \right) \leq \exp\left( -\frac{(b-x)^2}{C^2 T} \right).
    \]
\end{teo}
\begin{proof}
    We want to apply the concentration inequality result to $F = M_T.$ Recall that we have proved that $M_T \in \D^{1,2}$ and 
    \[
    DM_T = \sigma \sqrt{1-\rho^2} \1_{[0,\tau]}.
    \]
    Therefore,
    \[
    ||DM_T||_{\H} = \sqrt{\int_0^{\tau} (1-\rho^2) \sigma_t^2 dt} \leq C \sqrt{T}.
    \]
    If we define $m_T = \E[M_T]$ we obtain
    \[
    P(M_T - m_T \geq \lambda) \leq \exp\left( -\frac{\lambda^2}{C^2 T} \right).
    \]
    Therefore, if $b$ denotes the log-barrier we find that
    \[
    P(M_T \geq b) \leq \exp\left( -\frac{(b-m_T)^2}{C^2 T} \right).
    \]
    Finally, since $\lim_{T \to 0}m_T = x$ we find that there exists $T_0$ such that 
    \[
    (b-m_T)^2 \geq \frac{1}{2}(b-x)^2.
    \]
    Renaming the constants, we find that for $T \leq T_0$
    \[
     \P\left(M_T \geq b \right) \leq \exp\left( -\frac{(b-x)^2}{C^2 T} \right)
    \]
    as claimed.
\end{proof}

\section{Estimation Of The Density Of $M_T$} \label{sec 7}
In the previous section we have derived an estimation of the probability of attaining the barrier using an estimation of the cumulative distribution function of $M_T$. In this section we will proceed with more involved Malliavin calculus tools in order to explore if the previous bound can be improved. The objective of this section will be exploring the density function of $M_T$ and approximate the cumulative distribution function of $M_T$ using its density. Hence, we shall first show that the density exists and it is smooth. In order to check this property we will rely on the local criterion of existence and smoothness of densities for $\D^{1,2}$ random variables stated in Section \ref{sec 3}.

To be able to apply this criterion we need to find a set $A$, a stochastic process $u_A \in \D^{\infty}$ and a random variable $\gamma_A \in \D^{\infty}$ such that the equality $\langle DM_T, u_A\rangle = \gamma_A$ holds on the set $A$. We will rely on the process $Y$ to construct such $u_A$ and $\gamma_A$, since its behaviour is closely related to $M_T$ and it is Malliavin regular, as the following two results shows:

\begin{lema} \label{Y regular}
    For any $r \in [0,T]$, $Y_r \in \D^{\infty}$ and
    \[
    D^k Y_r = \int_{[0,r]^2} \frac{2p_0(2p_0-1) \cdots (2p_0 - k+ 1) (X_t- X_s)^{2p_0 - k}}{|t-s|^{\gamma_0}} D^k(X_t- X_s) dt ds.
    \]
\end{lema}
\begin{proof}
    The proof follows the same ideas as in \cite{pu2018stochastic}.
\end{proof}
We can in fact go further on the study of the regularity of $Y$. Indeed, not only the moments of $Y_r$ and its derivatives are finite for every $r \in [0,T]$, but they are uniformly bounded in $r \in [0,T]$ as we show in the upcoming lemma.
\begin{lema}
    For any $p \geq 1$ and for every integer $k$,
    \[
    \sup_{r \in [0,T]} \E\left[||D^kY_r||^p_{\H ^{\otimes k}}\right] < \infty.
    \]
\end{lema}
\begin{proof}
    We give the proof for $k = 1$. The proof for a general $k \geq 1$ uses Lemma \ref{Y regular} and follows the same lines as with the case $k = 1$. For the first order derivative we have
    \[
    DY_r = \int_{[0,r]^2} \frac{2p_0(X_t-X_s)^{2p_0 -1}}{|t-s|^{\gamma_0}} D(X_t-X_s) dt ds.
    \]
    Hence, Hölder's inequality for $p \geq 1$ shows that
    \begin{align*}
    \E\left[ ||DY_r||_{\H}^p \right] \leq c_p r^{2(p-1)} \int_{[0,r]^2} \frac{\E[|X_t-X_s|^{(2p_0-1)p}]}{|t-s|^{\gamma_0 p}} ||D(X_t-X_s)||^p_{\H}.
    \end{align*}
    Now, observe that
    \[
    D_s \sigma_t = \sigma_s \sqrt{1-\rho^2} \1_{[0,t]}(s),
    \]
    so
    \[
    D_u(X_t - X_s) = \sigma_u \sqrt{1-\rho^2} \1_{[t \wedge s, t \vee s]} (u).
    \]
    This implies that
    \[
    \int_0^T ( D_u(X_t - X_s))^2 du = (1-\rho^2)\int_{t \wedge s}^{t \vee s}  \sigma_u^2 du \leq C(1-\rho^2)|t-s|
    \]
    and therefore there exists a constant $C > 0$ such that
    \[
    ||D(X_t-X_s)||^p_{\H} \leq C|t-s|^p
    \]
    This estimate shows that
    \begin{align*}
     \E\left[ ||DY_r||_{\H}^p \right] \leq &c_p r^{2(p-1)}\int_{[0,r]^2} |t-s|^{(p_0-1/2)p + p - \gamma_0 p} dt ds \\
     \leq &c_p r^{2(p-1)}\int_{[0,r]^2} |t-s|^{p(p_0 - \gamma_0 + 1/2)} dt ds \\
     \leq & c_p r^{2(p-1)} r^{p_0p - p\gamma_0 + p/2 + 2} \\
     \leq &c_p r^{5p/2 + p_0p - p\gamma_0} \\
     \leq & c_p r^{p(p_0 - \gamma_0 + 5/2)}.
    \end{align*}
    The result follows from applying supremums to both sides of the inequality.
\end{proof}

We can now construct the setting in order to prove the existence and smoothness of the density. For all $z \in (x, \infty)$ let $a = \frac{x+z}{2}$ and $A = \left(a, \infty \right)$. Let $R_T := R_T(a)$ be the radius such that
\[
Y_r \leq R_T(a) \Rightarrow \sup_{t \in [0,r]} X_t \leq a.
\]Let $\psi_0(x)$ be a smooth function such that $\psi_0(x) = 0$ if $x > 1$, $\psi_0(x) = 0$ if $x \leq \frac{1}{2}$ and $\psi_0(x) \in [0,1]$ if $x \in \left[ \frac{1}{2}, 1\right]$. We define $\psi(x)$ as
\[
\psi(x) = \psi_0\left( \frac{x}{R_T} \right).
\]
Notice that $||\psi'||_{\infty} \leq cR^{-1}_T$. We define
\[
u_A = \frac{\psi(Y_{\cdot})}{\sigma_{\cdot} \sqrt{1-\rho^2}}, \quad \gamma_A = \int_0^T \psi(Y_r) dr.
\]
Now that we have the candidates of $A$, $u_A$ and $\gamma_A$ we have to prove that the criterion holds.
\begin{lema}
    We have $u_A \in \D^{\infty}(\H)$, $\gamma_A \in \D^{\infty}$ and $\langle DM_T, u_A\rangle = \gamma_A$ in the set $A$. In conclusion, $M_T$ has a smooth density.
\end{lema}
\begin{proof}
    The proof of the fact $u_A \in \D^{\infty}(\H)$ follows from the fact that $\sigma$ is independent of $B$, $Y \in \D^{\infty}(\H)$, the smoothness of $\psi$ and the chain rule of the Malliavin calculus. Similarly, the fact that $\gamma_A \in \D^{\infty}$ follows directly from the fact that $Y_r \in \D^{\infty}$, $\sup_{r \in [0,T]} ||D^k Y_r||_{\H^{\otimes k}}^p < \infty$ for all $k \geq 1$ and for all $p \geq 1$ and the chain rule of Malliavin calculus.

    Now, observe that
    \[
    \langle DM_T, u_A\rangle = \int_0^{\tau}  \sigma_t \sqrt{1-\rho^2} \frac{\psi(Y_t)}{\sigma_t \sqrt{1-\rho^2}} dt = \int_0^{\tau} \psi(Y_t) dt.
    \]
    We now need to prove that
    \[
    \int_0^{\tau} \psi(Y_t) dt = \int_0^T \psi(Y_t) dt
    \]
    or, equivalently, $\psi(Y_r) = 0$ if $r \in [\tau, T]$. Notice that if $\psi(Y_r) > 0$ if $r > \tau$ then we would have $Y_r \leq R_T$ and therefore, the following holds:
    \[
    M_T = X_{\tau} = \sup_{t \in [0,T]} X_t = \sup_{t \in [0,r]} X_t  \leq a.
    \]
    However, on the set $\{M_T \in A\}$ we have $M_T > a$ so we arrive at a contradiction, proving like this that
    \[
    \int_0^{\tau} \psi(Y_t) dt = \int_0^T \psi(Y_t) dt
    \]
    and therefore $\langle DM_T, u_A\rangle = \gamma_A$.
\end{proof}
This result, even though it's technical and does not tell us an explicit expression of the density function of $M_T$, it allows us to obtain a representation that we will use later on to obtain an estimate.
\begin{prop}
    The probability density function of $M_T$ at the point $z \in (x, \infty)$ is given by
    \[
    p(z) = \E\left[\1_{\{M_T > z\}} \delta\left( \frac{u_A}{\gamma_A}\right)\right]
    \]
\end{prop}
\begin{proof}
    Let $\kappa : \R \to [0,1]$ be an infinitely differentiable function such that
    \[
    \kappa(r) = \begin{cases}
        0 & r\leq \frac{x+z}{2}, \\
        1 & r \geq \frac{x+3z}{4}.
    \end{cases}
    \]We define also the random variable $G = \kappa(M_T)$. Observe that $G = 0$ on the set $\{M_T \notin A\}$. Let $f \in \mathcal{C}^{\infty}_0(\R)$ and let $\phy(x) = \int_{-\infty}^x f(y) dy$. On the set $\{M_T \in A\}$ we have
    \[
    \langle D\phy(M_T), u_A\rangle = \phy'(M_T) \langle DM_T, u_A\rangle = \phy'(M_T) \gamma_A
    \]
    and therefore
    \[
    G\phy'(M_T) = \langle D\phy(M_T), \frac{Gu_A}{\gamma_A}\rangle
    \]
    Therefore, if we take expectations on both sides and we use the duality relationship between the Malliavin derivative and the divergence operator we have
    \[
    \E[G\phy'(M_T)] = \E\left[\phy(M_T) \delta \left( \frac{Gu_A}{\gamma_A}\right) \right].
    \]
    Hence,
    \[
    \E[G \phy'(M_T)] = \int_{\R} f(y) \E\left[ \1_{ \{M_T > y\} } \delta \left( \frac{Gu_A}{\gamma_A}\right)\right] dy.
    \]
    Finally, since $\phy' = f$ we find that for every $y \in (\frac{x+3z}{4}, \infty)$ it is satisfied that $G = 1$ and therefore the density at the point $z \in (\frac{x+3z}{4}, \infty)$ is given by
    \[
    p(z) = \E\left[ \1_{ \{M_T > z\} } \delta \left( \frac{u_A}{\gamma_A}\right)\right].
    \]
\end{proof}
The Cauchy-Schwartz inequality provides a direct upper bound for $p(z)$. Indeed,
\[
p(z) \leq \P (M_T \geq z)^{1/2} ||\delta(u_A \gamma_A^{-1})||_{L^2(\Omega)}.
\]
The tail estimate obtained in the previous section allows us to deal with the $\P (M_T \geq z)^{1/2}$ term, so the analysis of the density of $M_T$ is focused on the analysis of the term $||\delta(u_A \gamma_A^{-1})||_{L^2(\Omega)}$. Let's break down the strategy to analyse this latter term. Proposition \ref{prop: integral of product} allows us to write
\[
\delta\left( \frac{u_A}{\gamma_A}\right) = \frac{\delta(u_A)}{\gamma_A} + \frac{\langle D\gamma_A, u_A\rangle}{\gamma_A^2} = I_1 + I_2,
\]
so
\[
\left| \left|\delta\left( \frac{u_A}{\gamma_A}\right) \right| \right|_{L^2(\Omega)} \leq ||I_1||_{L^2(\Omega)} + ||I_2||_{L^2(\Omega)}.
\]
The following sequence of lemmas provide an analysis of the moments of the terms involved.
\begin{lema} There exists a constant $C > 0$ depending on $\psi$, $p$, $\sigma$ and $\rho$ such that
    \[
    ||\delta(u_A)||_{L^p(\Omega)} \leq C \sqrt{T}
    \]
\end{lema}
\begin{proof}
    It is clear that 
    \[
    \delta(u_A) = \int_0^T \frac{\psi(Y_r)}{\sigma_r \sqrt{1-\rho^2}} dB_r
    \]
    so Burkholder's inequality, the definition of $\psi$ and the hypotheses on $\sigma$ show that
    \begin{align*}
    \E[|\delta(u_A)|^p] \leq & c_p\E \left[ \left( \int_0^T \frac{\psi(Y_r)^2}{\sigma_r^2(1-\rho^2)} dr\right)^{p/2}\right] \\
    \leq & C \E\left[ \left( \int_0^T \psi(Y_r)^2 dr\right)^{p/2}\right] \\
    \leq & C T^{p/2}
    \end{align*}
    so 
    \[
    ||\delta(u_A)||_{L^p(\Omega)} \leq C\sqrt{T}.
    \]
\end{proof}

\begin{lema} \label{aux 1} There exists a constant $C > 0$ depending on $\psi$, $\rho$ and $\sigma$ such that
    \[
    ||u_A||_{\H} \leq C \sqrt{T}.
    \]
\end{lema}

\begin{proof}
    Applying the definition of $u_A$ we see that
    \[
    ||u_A||_{\H}^2 = \int_0^T \frac{\psi(Y_r)^2}{\sigma_r^2 \sqrt{1-\rho^2}} dr
    \]
    so the conclusion follows using the same estimations as in the previous lemma.
\end{proof}
\begin{lema} \label{aux 2} There exists a constant $C > 0$ depending on $\psi$, $\rho$ and $\sigma$ such that
    \[
    ||\langle D\gamma_A , u_A \rangle_{\H} ||_{L^p(\Omega)} \leq CT^4
    \]
    for some constant $C > 0$ not depending on $T$.
\end{lema}
\begin{proof}
    Since $D_t \gamma_A = \int_0^T \psi'(Y_r) D_u Y_r dr$ we have that
    \[
    \langle D\gamma_A, u_A \rangle = \int_0^T \int_0^T \psi'(Y_r) D_t Y_r  u_A(t) dr dt = \int_0^T \psi'(Y_r) \langle DY_r, u_A\rangle dr.
    \]
    Now, Hölder's inequality for $p \geq 1$ shows that
    \begin{align*}
    \E[|\langle D\gamma_A, u_A\rangle|^p] \leq &T^{p-1} ||\psi'||_{\infty}^p\int_0^T \E[|\langle DY_r, u_A\rangle|^p] dr \\
    \leq & C T^{p-1} R_T(a)^{-p} \int_0^T \E[||DY_r||_{\H}^p ||u_A||_{\H}^p] dr \\
    \leq & C T^{3p/2-1}R_T(a)^{-p} \int_0^T \E[||DY_r||_{\H}^p] dr \\
    \leq & C R_T(a)^{-p} T^{3p/2-1} \int_0^T  r^{p(p_0 - \gamma_0 + 5/2)} dr \\
    \leq & C R_T(a)^{-p} T^{4p+pp_0 - p\gamma_0 -2p+\frac{p\gamma_0}{2}} \\
    \leq & C T^{p(2+p_0-\frac{\gamma_0}{2})} \\
    \leq & C T^{4p}
    \end{align*}
    where we have applied the bound for $\E[||DY_r||_{\H}^p]$, the definition of $R_T(a)$ and the fact that $p_0 > 6$.
\end{proof}
\begin{lema} \label{aux 3}
    The random variable $\gamma_A$ has negative moments of all orders- Furthermore, if we assume $z \geq x + T^{1/4}$ then there exists a constant $C > 0$ not depending on $T$ such that
    \[
    ||\gamma_A^{-1}||_{L^p(\Omega)} \leq \frac{C}{T}
    \]
\end{lema}
\begin{proof}
    Since
    \[
    \gamma_A = \int_0^T \psi(Y_r) dr \geq \int_0^T \1_{\{ Y_r \le \frac{R_T(a)}{2}\}} dr =: \Gamma_A.
    \]
    Let $0<\eps < T$. We find that
    \[
    \P(\Gamma_A < \eps) \leq \P(Y_{\eps} \geq R_T(a)/2)\leq \frac{2^q \E[Y_{\eps}^q]}{R_T(a)^q} \leq C_q \eps^{2q} R_T(a)^{-q}  T^{(p_0 - \gamma_0)q}.
    \]
    This allows us to conclude, thanks to \cite{dalang2009minicourse} (Chapter 3, Lemma 4.4), that $\Gamma_A$ has negative moments of all orders and therefore so does $\gamma_A$. We can make a further analysis to show that for every $2q > p$ we have
    \begin{align*}
        \E[\Gamma_A^{-p}] = &p \int_0^{\infty} y^{p-1} \P(\Gamma_A^{-1} > y) dy \\
        = & p\int_0^{T^{-1}} y^{p-1} \P(\Gamma_A^{-1} > y) dy + p\int_{T^{-1}}^{\infty} y^{p-1} \P(\Gamma_A^{-1} > y) dy \\
        \leq & CT^{-p} + CR_{T}(a)^{-q}T^{(p_0 - \gamma_0)q}\int_{T^{-1}}^{\infty} y^{p-2q-1} dy. \\
    \end{align*}
    Since
    \[
    R_T(a) = C (a-x)^{-2p_0}T^{\frac{4-\gamma_0}{2}}
    \]
    and $a-x > \frac{T^{1/4}}{2}$ because $z > x+ T^{1/4}$, then
    \[
    R_T(a)^{-q} \leq T^{-\frac{qp_0}{2}}T^{\frac{-(4-\gamma_0)q}{2}}.
    \]
    This implies that
    \[
     R_{T}(a)^{-q}T^{(p_0 - \gamma_0)q}\int_{T^{-1}}^{\infty} y^{p-2q-1} dy \leq T^{-\frac{qp_0}{2}}T^{\frac{-(4-\gamma_0)q}{2}}T^{(p_0 - \gamma_0)q}T^{2q-p}.
    \]
    We gather together the exponents as
    \begin{align*}
        -\frac{qp_0}{2}-2q+\frac{\gamma_0q}{2} + qp_0 -\gamma_0 q + 2q -p = & \frac{qp_0}{2} - \frac{q\gamma_0}{2} -p
    \end{align*}
    to obtain that
    \[
    CT^{-p} + CR_{T}(a)^{-q}T^{(p_0 - \gamma_0)q}\int_{T^{-1}}^{\infty} y^{p-2q-1} dy \leq CT^{-p}(1+T^{\frac{p_0 - \gamma_0}{2}}).
    \]
    Finally, using the relationships $p_0 - 2 > \gamma_0 > 4$ we conclude that the exponent $\frac{p_0 - \gamma_0}{2}$ is positive. Consequently, there exists $C > 0$ such that
    \[
     \E[\Gamma_A^{-p}] \leq CT^{-p}.
    \]
    since $\gamma_A \geq \Gamma_A$ we conclude the result.
\end{proof}
We are now ready to prove the upper bound for the density of $M_T$.
\begin{teo} \label{th: bound density}
    There exists $T_0 > 0$ such that for all $T \leq T_0$, the density function $p(z)$ of the random variable $M_T$ satisfies the following upper bound:
    \[
    p(z) \leq2C \frac{1}{\sqrt{T}}\exp\left( -\frac{(z-x)^2}{2C^2 T}\right)
    \]
    for a constant $C > 0$ not depending on $T$.
\end{teo}
\begin{proof}
    As it was shown,
    \[
    p_0(z) \leq \P(M_T \geq z)^{1/2} (||I_1||_{L^2(\Omega)} + ||I_2||_{L^2(\Omega)}).
    \]
    Now, using Hölder's inequality we see that
    \[
    ||I_1||_{L^2(\Omega)} \leq C T^{1/2} T^{-1} = C T^{-1/2}, \quad ||I_2||_{L^2(\Omega)} \leq C T^{4} T^{-2} \leq CT^2.
    \]
    Hence, plugging all the estimates obtained in Lemmas \ref{aux 1}, \ref{aux 2} and \ref{aux 3} we are able to show that there exist $c_1, c_2 > 0$ such that
    \[
    p(z) \leq c_1\left(\frac{1}{\sqrt{T}} + T^2 \right) \exp\left( -\frac{(z-m_T)^2}{2c_2^2 T}\right) \leq 2c_1 \frac{1}{\sqrt{T}}\exp\left( -\frac{(z-m_T)^2}{2c_2^2 T}\right).
    \]
    Using again that $m_T \to x$ as $T \to 0$ and renaming the constants we obtain the claimed statement.
\end{proof}
This result is the key to understand the asymptotic behaviour of the Call price $C_0^b$.
\begin{teo}
    Let $C_0^b$ be an up-and-in barrier call option with log-barrier $b > x = \log S_0$. Then, the probability of reaching the barrier satisfies the upper bound
    \[
    \P(M_T \geq b) \leq \int_{b}^{\infty}\frac{c_1}{\sqrt{T}} \exp\left( -\frac{(z-x)^2}{2c_2 T}\right) dz
    \]
    and, in particular, $C_0^b$ approaches zero faster than any polynomial of $T$.
\end{teo}
\begin{proof}
    The fact that
    \[
    \P(M_T \geq b) \leq \int_{b}^{\infty}\frac{c_1}{\sqrt{T}} \exp\left( -\frac{(z-x)^2}{2c_2 T}\right) dz
    \]
    comes directly from Theorem \ref{th: bound density}. As a direct consequence we find that, for every $\alpha > 0$,
    \[
    \lim_{T \to 0} \frac{\P(M_T \geq b)}{ T^{\alpha}} = 0.
    \]
    The result now follows from applying Hölder's inequality in $C_0^{b}$.
\end{proof}
\section{Application To The Rough Bergomi Model} \label{sec 8}
One of the main drawbacks of the main result we have given is that we are assuming that there exist two constants $\alpha, \beta$ such that $\alpha < \sigma_t < \beta$ and we have heavily relied on this hypothesis in order to derive the upper bound for $\P(\sup_{t \in [0,T]} S_t \geq B)$. However, using an approximation argument, we aim to show that $\P(\sup_{t \in [0,T]} S_t \geq B)$ approaches zero faster than any polynomial even though we don't have an explicit estimation of such probability.

It is difficult to construct an approximation argument that works model-independently. Indeed, it is also hard to construct general approximation arguments even when studying European call options. We aim, therefore, to extend the result for one of the most popular rough volatility models. In this section we will apply the main result of this paper to prove that the faster-than-polynomial speed of convergence holds for the Rough Bergomi model. In this specific scenario we are assuming that the log-price is of the form
\[
X_t = x - \frac{1}{2}\int_0^t \sigma_s^2 ds + \int_0^t \sigma_s dZ_s
\]
where $Z_t = \rho W_t + \sqrt{1-\rho^2} B_t$ and $\sigma$ has the form
\[
\sigma_t^2 = \sigma_0^2 \exp\left( \nu W^H_t - \frac{\nu^2 t^{2H}}{2}\right), \quad W^H_t = \sqrt{2H}\int_0^t (t-s)^{H-1/2} dW_s.
\]
Notice that in this case we can't apply the main result of this paper directly to $\sigma$ because it is not true that there exist two constants $\alpha, \beta > 0$ such that $\alpha < \sigma_t < \beta$ almost surely. However, we can use a truncation argument to show that the speed of convergence of the probability of reaching the barrier is still faster than any polynomial. In order to do so, we consider the truncated log-price
\[
X_t^n = x - \frac{1}{2}\int_0^t (\sigma_s^n)^2 ds + \int_0^t \sigma_s^n dZ_s
\]
where 
\[
\sigma^n_t = \phi_n \left(\nu W^H_t - \frac{\nu^2 t^{2H}}{2} \right)
\]
and $\phi_n(x)$ is defined as follows: consider $\phi(x) = \sigma_0\exp(x)$. For every $n > 1$ we consider $\phi_n \in \mathcal{C}^2_b$ satisfying $\phi_n(x) = \phi(x)$ for all $x \in [-n,n]$, $\phi_n(x) \in [\phi(-2n) \vee \phi(x), \phi(-n)]$ for all $x \leq -n$ and $\phi_n(x) \in [\phi(n), \phi(x) \vee \phi(2n)]$ if $x \geq n$. We also define $S^n_t = \exp(X_t^n)$. Notice that for every fixed $\eps > 0$ 
\[
\P(\sup_{t \in [0,T]} S_t \geq B) \leq \P(\sup_{t \in [0,T]} S_t^n \geq B- \eps) + \P(\sup_{t \in [0,T]} |S_t^n - S_t| \geq \eps) =: A_1 + A_2.
\]
We aim to show that for every fixed $m > 0$, $\lim_{T \to 0} T^{-m}\P(\sup_{t \in [0,T]} S_t \geq B)= 0$. Notice that, due to the main result of this paper, we have that
\[
\lim_{T \to 0} T^{-m} A_1 = \lim_{T \to 0} T^{-m}\P(\sup_{t \in [0,T]} S^n_t \geq B)= 0.
\]
Regarding $A_2$, we use Chebyshev's inequality and Doob's martingale inequality to show that
\begin{align*}
    A_2 \leq &\frac{ \E\left[ |\sup_{t \in [0,T]} S_t - \sup_{t \in [0,T]} S_t^n |^p\right]}{\eps^p} \\ 
    \leq & \frac{ \E\left[ \sup_{t \in [0,T]} |S_t - S_t^n|^p\right]}{\eps^p} \\
    \leq & \frac{1}{\eps^p} \left( \frac{p}{p-1} \right)^p \E[|S_T - S_T^n|^p].
\end{align*}
Finally, using the same argument as in \cite{alos2019estimating}, we can find $p \geq 2$ large enough such that $\lim_{T \to 0} T^{-m} A_2 = 0$. Therefore, for every fixed $m > 0$ we find that
\[
\lim_{T \to 0} \frac{\P(\sup_{t \in [0,T]} S_t \geq B)}{T^m} = 0
\]
so $C_0^b$ approaches zero faster than any polynomial even though the bound for $\P(\sup_{t \in [0,T]} S_t \geq B)$ does not necessarily hold if the hypothesis $\alpha < \sigma_t < \beta$ fails.

In order to numerically visualize the difference between the asymptotic behaviours of both European call options and up-and-in barrier call options under the Rough Bergomi model, we will perform three different simulations: one in the In-The-Money case (ITM), one in the At-The-Money case (ATM) and the last one in the Out-The-Money case (OTM).

\begin{ex}
    Consider $S$ following Model \eqref{model} with $S_0 = 10$ and $K = 9.5$, where the parameters of the model are $\rho = -0.3$, $\nu = 0.5$ and $H = 0.2$. In this case, asymptotic behaviour     of both options are displayed in Figure \ref{fig: ITM}.
    \begin{figure}[H]
        \centering
        \includegraphics[width=0.5\linewidth]{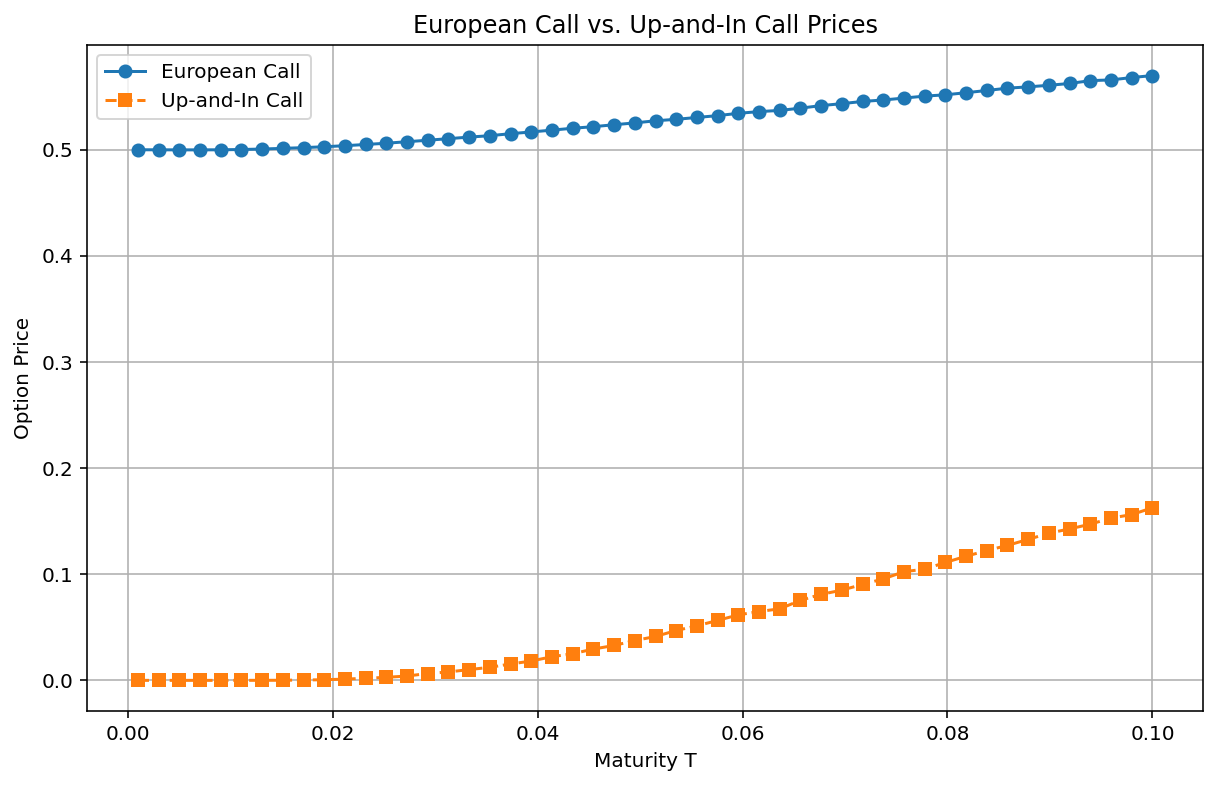}
        \caption{ITM European and up-and-in call options}
        \label{fig: ITM}
    \end{figure}
    Notice that the price of the up-and-in barrier option tends to zero, while the price of the European call option tends to $S_0 - K = 0.5$. Hence, the numerics are consistent with \cite{jafari2025option} and Theorem \ref{th: main theorem}.
\end{ex}

\begin{ex} \label{ex: ATM}
    Consider $S$ following Model \eqref{model} with $S_0 = K = 10$ where the parameters of the model are $\rho = -0.3$, $\nu = 0.5$ and $H = 0.2$. In this case, the asymptotic behaviour of both options are displayed in Figure \ref{fig: ATM}.
     \begin{figure}[H]
        \centering
        \includegraphics[width=0.5\linewidth]{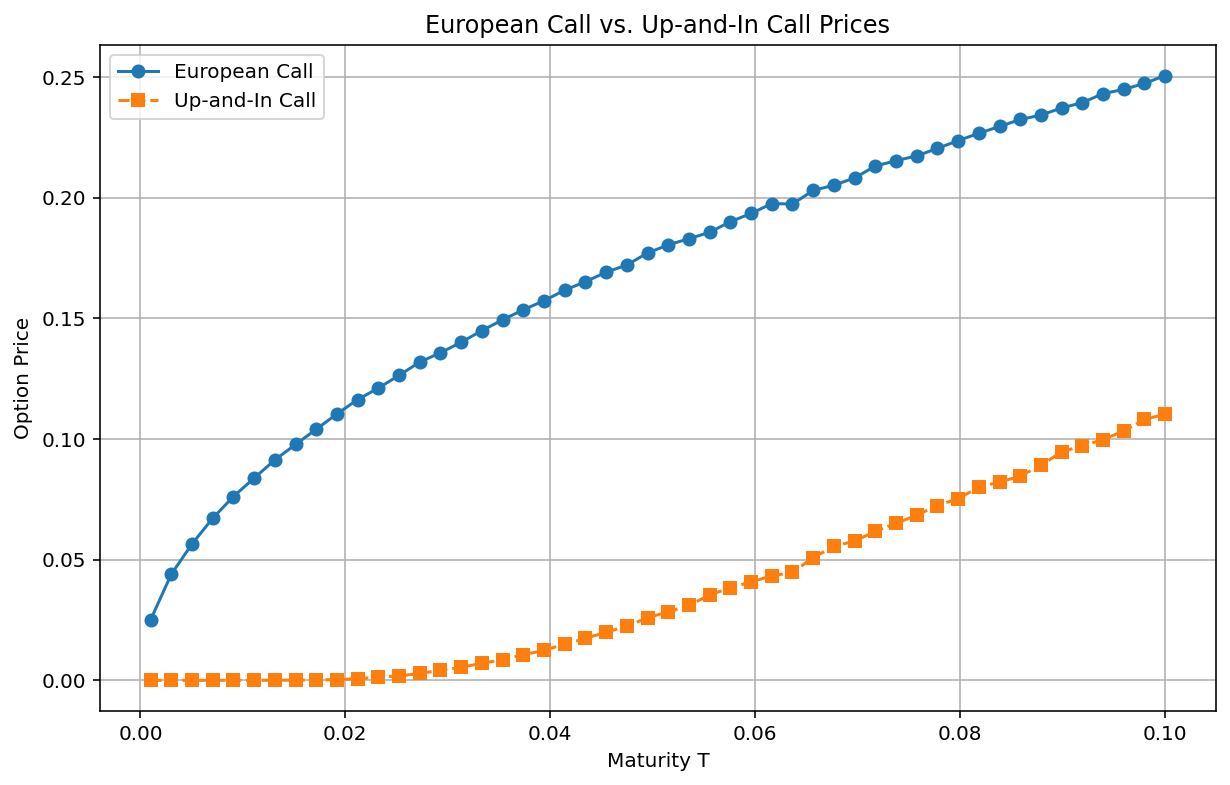}
        \caption{ATM European and up-and-in call options}
        \label{fig: ATM}
    \end{figure}
\end{ex}

\begin{ex} \label{ex: OTM}
    Consider $S$ following Model \eqref{model} with $S_0 = 10$, $K = 11$ where the parameters of the model are $\rho = -0.3$, $\nu = 0.5$ and $H = 0.2$. In this case, the asymptotic behaviour of both options are displayed in Figure \ref{fig: OTM}.
     \begin{figure}[H]
        \centering
        \includegraphics[width=0.5\linewidth]{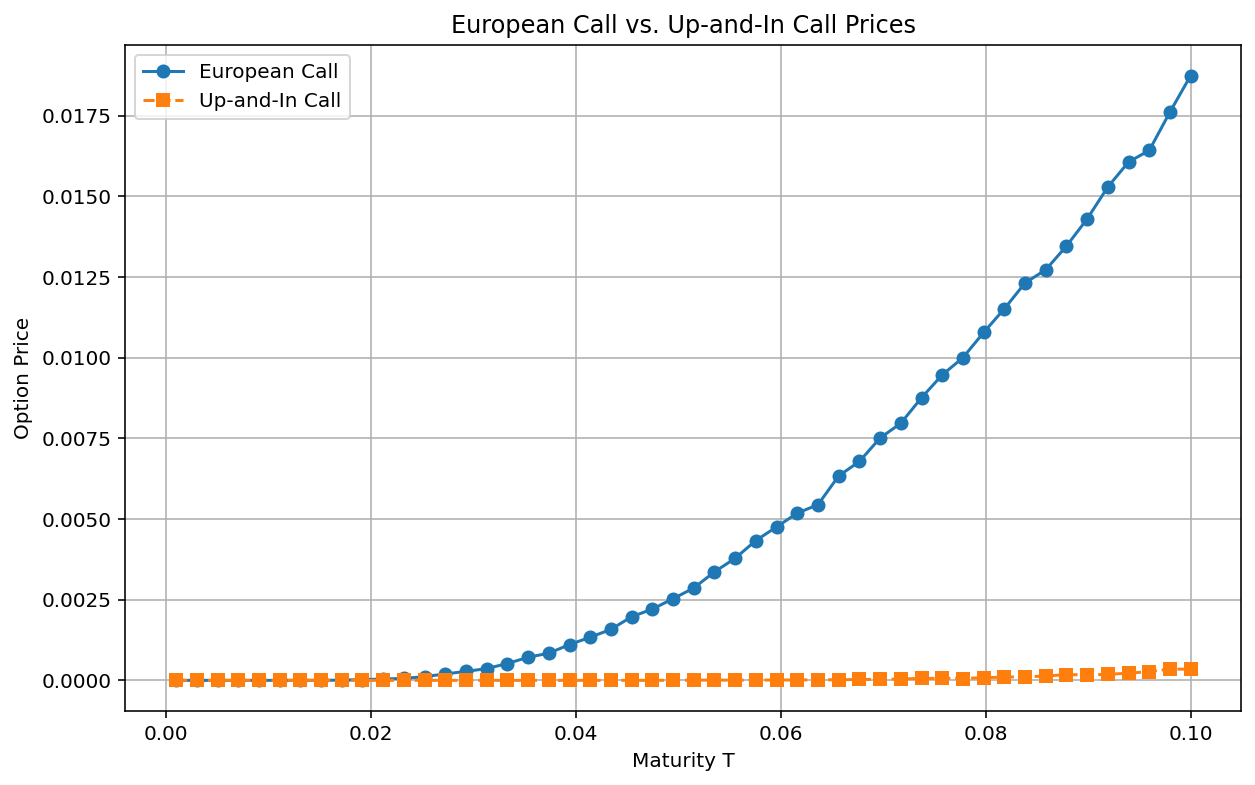}
        \caption{ATM European and up-and-in call options}
        \label{fig: OTM}
    \end{figure}
\end{ex}

Observe that in both Examples \ref{ex: ATM} and \ref{ex: OTM} both options converge to zero as $T \to 0$ as it is stated in \cite{jafari2025option} and Theorem \ref{th: main theorem}. Moreover, it is clearly visible that the speed of convergence of the up-and-in barrier call option is faster than the European one. 

\section{Conclusions}

In this paper, we employed Malliavin calculus techniques to analyse the asymptotic behaviour of up-and-in barrier call options under a general stochastic volatility model. Drawing inspiration from the methodologies developed in \cite{nualart1988continuite} and \cite{florit1995local} for Gaussian processes, as well as \cite{dalang2020density} and \cite{pu2018stochastic} for the stochastic heat equation with additive noise, we studied the distribution of the supremum of an asset price governed by a general stochastic volatility process. These probabilistic insights were then applied to characterize the behaviour of up-and-in barrier call option prices, concluding in this way that the rate of convergence to zero of up-and-in barrier call options is faster than any polynomial of the time to maturity $T$.  In particular, under the Rough Bergomi model, we quantified the rate of decay and provided numerical evidence highlighting the faster convergence of barrier option prices compared to their European counterparts.

\section*{Acknowledgements}
I would like to express my sincere gratitude to Elisa Alòs and Josep Vives for their valuable discussions, insightful suggestions and careful reading of the manuscript.

\section*{Funding}
Òscar Burés supported by program AGAUR-FI ajuts (2025 FI-1 00580) from the Department of Research and Universities of the Government of Catalonia and the co-funding of the European Social Fund Plus (ESF+).

\section*{Conflict of interest}
The author declares there is no conflict of interest.
\bibliographystyle{apalike}
\bibliography{refs.bib}
\end{document}